\documentclass[a4paper,draft,reqno,11pt]{amsart}
\usepackage[english]{babel}
\usepackage{amsmath}
\usepackage{amssymb}
\usepackage{amscd}
\usepackage{amsthm}
\usepackage{euscript}
\usepackage{tikz}
\newtheorem{proposition}{Proposition}
\newtheorem{lemma}{Lemma}
\newtheorem{theorem}{Theorem}

\theoremstyle{definition}
\newtheorem{definition}{Definition}

\theoremstyle{remark}

\begin{document}

\date{}
\title[Rigid trinomial varieties]{Rigid trinomial varieties}
\author{Polina Evdokimova, Sergey Gaifullin  and Anton Shafarevich}

\address{Lomonosov Moscow State University, Faculty of Mechanics and Mathematics, Department of Higher Algebra, Leninskie Gory 1, Moscow, 119991, Russia}
\email{polina.evdokimova@math.msu.ru}

\address{Lomonosov Moscow State University, Faculty of Mechanics and Mathematics, Department of Higher Algebra, Leninskie Gory 1, Moscow, 119991, Russia; \linebreak
Moscow Center of Fundamental and Applied Mathematics,  Moscow, Russia; \linebreak and \linebreak
Faculty of Computer Science, HSE University, Pokrovsky Boulevard 11, Moscow, 109028 Russia}
\email{sgayf@yandex.ru}

\address{Moscow Center of Fundamental and Applied Mathematics,  Moscow, Russia; \linebreak and \linebreak
Faculty of Computer Science, HSE University, Pokrovsky Boulevard 11, Moscow, 109028 Russia}
\email{shafarevich.a@gmail.com}

\thanks{The work was supported by the Foundation for the Advancement of Theoretical Physics and Mathematics “BASIS”.}

\maketitle

\begin{abstract}
An algebraic variety $X$ is called rigid if there is no non-trivial action on $X$ of the additive group of the base field. A trinomial variety is an affine variety that is given by a set of equations consisting of polynomials with three monomials; see Definition \ref{defTrin}. In this paper, we complete the classification of rigid trinomial varieties started in \cite{IA} and \cite{SG}. 

\end{abstract}

\section{Introduction}

Let $\mathbb{K}$ be an algebraically closed field of characteristic zero, $\mathbb{G}_a$ be the additive group of the field $\mathbb{K}$, and $X$ be an affine irreducible algebraic variety over $\mathbb{K}$. It is always natural to study actions of algebraic groups on $X$. In particular, one can consider $\mathbb{G}_a$-actions on $X$. 

There is a way to study $\mathbb{G}_a$-actions algebraically using the theory of locally nilpotent derivations. Let $R$ be a commutative algebra over $\mathbb{K}$. A linear map $\partial:R \to R$ is called a \emph{derivation} if it satisfies the Leibniz rule: $\partial(ab) = a\partial(b) + b\partial(a)$ for all $a,b \in R$. We say that a derivation $\partial$ on $R$ is \emph{locally nilpotent} if for any $a\in R$ there is a positive integer $n$ such that $\partial^n(a) = 0$. We denote by $\mathrm{LND}(R)$ the set of all locally nilpotent derivations on $R$. If $X$ is an affine variety and $\mathbb{K}[X]$ is an algebra of regular functions on $X$ then by a locally nilpotent derivation on $X$ we mean a locally nilpotent derivation on $\mathbb{K}[X]$ and by $\mathrm{LND}(X)$ we denote the set of locally nilpotent derivations on $\mathbb{K}[X]$. 

The theory of locally nilpotent derivations is well studied; see \cite{GF}. A locally nilpotent derivation $\partial$ on $X$ defines a $\mathbb{G}_a$-action on $X$ as follows. An element $t \in \mathbb{G}_a$ acts on $\mathbb{K}[X]$ by the following rule:

$$t\circ f = \mathrm{exp}(t\partial)(f) = \sum_{i = 0}^{\infty} \frac{t^i\partial^i(f)}{i!}.$$
Since $\partial$ is locally nilpotent, the last sum has only a finite number of non-zero terms. This map establishes a bijection between locally nilpotent derivations on $X$ and $\mathrm{G}_a$-actions on $X$; see \cite[Section 1.5]{GF}.

We say that a variety $X$ is \emph{rigid} if there is no non-trivial $\mathbb{G}_a$-action on $X$, or equivalently, $\mathrm{LND}(X) = \{0\}$. Examples of rigid varieties can be found in \cite{CM, KZ, CPW} and \cite[Chapter 10]{GF}. It was proven in \cite{AG} that the automorphism group of a rigid variety contains a unique maximal algebraic torus. In some cases, this allows to describe all automorphisms of rigid varieties, which is usually difficult to do for affine varieties.  One can find examples in \cite{AG, PZ, ST}. 

In this paper, we study \emph{trinomial varieties}. These are varieties given by systems of polynomial equations of the form

$$c_0T_{01}^{l_{01}}\ldots T_{0n_0}^{l_{0n_0}} + c_1T_{11}^{l_{11}}\ldots T_{1n_1}^{l_{1n_1}} + c_2T_{21}^{l_{21}}\ldots T_{2n_2}^{l_{2n_2}} = 0, $$
where $T_{ij}$ are variables, $c_i \in \mathbb{K}\setminus\{0\}$, $n_0\geq 0, n_1, n_2 \geq 1$, and $l_{ij}$ are non-negative integers.  The coefficients $c_i$ must satisfy certain conditions; see Definition \ref{defTrin} or \cite[Construction 1.1]{HW} for a precise definition. 

Every trinomial variety admits an action of an algebraic torus of complexity one. This means that there is an action of a torus such that a generic orbit has codimension 1. Moreover, every normal, rational variety $X$ with only constant invertible functions, finitely generated divisor class group and an algebraic torus action of complexity one can be obtained as a quotient of a trinomial variety via an action of a diagonalizable group; see \cite[Corollary 1.9]{HW}.

In \cite{CPW}, it was proven that the Pham-Brieskorn surface
$$x_0^{k} + x_1^l + x_2^m = 0,\ m \geq l \geq k \geq 2,$$
which is a trinomial hypersurface, is not rigid if and only if $k = l = 2$. In \cite{IA} a criterion of rigidity of a factorial trinomial hypersurface was obtained. In \cite{SG} this result was extended to arbitrary trinomial hypersurfaces, factorial trinomial varieties and all trinomial varieties of Type 1. In this paper, we give a criterion of rigidity of an arbitrary trinomial variety of Type 2 (Theorem \ref{Type2}). This implies a criterion of rigidity of an arbitrary trinomial variety (Theorem \ref{MainTheorem}).

\section{Trinomial varieties}

Here we recall the definition of trinomial varieties.

\begin{definition}\label{defTrin}\cite[Construction 1.1]{HW}

Fix integers $r, n > 0, m \geq 0$ and $q \in \{0, 1\}$. Also, fix a partition 
$$n = n_q + \ldots + n_r, \ n_i >0.$$
For each $i = q, \ldots, r$, fix a tuple $l_i = (l_{i1}, \ldots, l_{in_i})$ of positive integers and define a monomial
$$T_i^{l_i} = T_{i1}^{l_{i1}}\ldots T_{in_i}^{l_{in_i}} \in \mathbb{K}[T_{ij}, S_k | q \leq i \leq r, 1 \leq j \leq n_i, 1 \leq k \leq m].$$

We write $\mathbb{K}[T_{ij}, S_k]$ for the above polynomial ring. Now we define a ring $R(A)$ for some input data $A$. 

\emph{Type 1.} $q = 1, A = (a_1, \ldots, a_r)$ where $a_j \in \mathbb{K}$ with $a_i \neq a_j$ for $i\neq j$. Set $I = \{1, \ldots, r-1\}$ and for every $i \in I$ define a polynomial
$$g_i = T_i^{l_i} - T_{i+1}^{l_{i+1}} - (a_{i+1} - a_i) \in \mathbb{K}[T_{ij}, S_k].$$

\emph{Type 2.} $q = 0,$

$$A = \begin{pmatrix} 

a_{00} & a_{01} & a_{02} & \ldots & a_{0r}\\
a_{10} & a_{11} & a_{12} & \ldots & a_{1r}

\end{pmatrix}$$
is a $2 \times (r+1)$-matrix with pairwise linearly independent columns. Set $I = \{0, \ldots, r-2\}$ and for every $i \in I$ define a polynomial 
$$g_i = \mathrm{det}\begin{pmatrix} 
T_{i}^{l_i} & T_{i+1}^{l_{i+1}} & T_{i+2}^{l_{i+2}} \\
a_{0i} & a_{0i+1} & a_{0i+2} \\
a_{1i} & a_{1i+1} & a_{1i+2}
\end{pmatrix} \in \mathbb{K}[T_{ij}, S_k].$$

For both types we define $R(A) = \mathbb{K}[T_{ij}, S_k]/(g_i \mid i \in I).$

\end{definition}

If $m>0$ then trinomial varieties of Type 1 and Type 2 are not rigid. Indeed, we have a non-zero locally nilpotent derivation $\frac{\partial}{\partial S_1}$ on $R(A)$. 

The following criterion for rigidity of trinomial varieties of Type 1 was obtained in \cite{SG}.

\begin{theorem}\cite[Theorem 3]{SG}\label{Type1} Let $X$ be a trinomial variety of Type 1. Then $X$ is not rigid if and only if one of the following holds:
\begin{enumerate}
\item $m > 0$;
\item There is $b \in \{1,\ldots, r\}$ such that for each $ i\in \{1, \ldots, r\} \setminus \{b\}$ there is $j(i) \in \{1, \ldots, n_i\}$ with $l_{ij(i)} = 1$.

\end{enumerate}

\end{theorem} 

We denote $\mathfrak{l}_i = \mathrm{gcd}(l_{i1}, \ldots, l_{in_i}).$ We will need the following criterion of rationality of trinomial varieties of Type 2.

\begin{theorem}\cite[Corollary 5.8]{ABHW}\label{rational}
A trinomial variety $X$ is rational if and only if one of the following holds

\begin{enumerate}
\item $\mathrm{gcd}(\mathfrak{l}_i, \mathfrak{l}_j) = 1$ for all $i, j \in \{0, \ldots, r\}$

\item There are $i, j \in \{0, \ldots, r\}$ such that $\mathrm{gcd}(\mathfrak{l}_i, \mathfrak{l}_j) > 1$ and for all pairs $(k, s) \neq (i, j)$ we have $\mathrm{gcd}(\mathfrak{l}_k, \mathfrak{l}_s) = 1$.

\item There are $i, j, k \in \{0,\ldots, r\}$ with $\mathrm{gcd}(\mathfrak{l}_i, \mathfrak{l}_j) = \mathrm{gcd}(\mathfrak{l}_i, \mathfrak{l}_k) = \mathrm{gcd}(\mathfrak{l}_j, \mathfrak{l}_k) = 2$ and $\mathrm{gcd}(\mathfrak{l}_u, \mathfrak{l}_v) = 1$ for all $u, v \in \{0, \ldots, r\}, \{u, v\} \nsubseteq \{i, j, k\}$.
\end{enumerate}
\end{theorem}

\section{Suspensions}\label{Susp}

As in \cite{SG}, one of the key ideas of the proof of Theorem \ref{MainTheorem} is that one can consider trinomial varieties as suspensions over trinomial varieties of lower dimension. 

\begin{definition}
Let $X$ be an affine variety and $m$ be a positive integer.  Given a non-constant regular function $f \in \mathbb{K}[X]$ and positive integers $k_1, \ldots, k_m$, we define a variety 
$$Y = \mathrm{Susp}(X, f, k_1, \ldots, k_m) = \mathbb{V}(y_1^{k_1}\ldots y_m^{k_m} - f(x)) \subseteq \mathbb{K}^m \times X,$$
called an \emph{m-suspension} over $X$ with weights $k_1, \ldots, k_m$. 
\end{definition}

One can consider the field $\mathbb{L}_i = \overline{\mathbb{K}(y_i)}$ that is the algebraic closure of $\mathbb{K}(y_i).$ We denote by $Y_i = Y(\mathbb{L}_i)$ the affine variety over $\mathbb{L}_i$ that is given by the same equations as $Y$. We will use the following lemma.

\begin{lemma}\cite[Lemma 7]{SG} Let $1 \leq i \neq j \leq m$. Assume $Y_i$ and $Y_j$ are rigid varieties. Then $Y$ is rigid.

\end{lemma}

\section{Trinomial varieties of Type 2}

In this section, we prove the main theorem.

\begin{theorem}\label{Type2}

Let $X$ be a trinomial variety of Type 2. Then $X$ is not rigid if and only if one of the following holds:
\begin{enumerate}

\item $m > 0$;

\item There are at most two numbers $a, b \in \{0,\ldots, r\}$ such that for each $i \in \{0,\ldots, r\} \setminus \{a, b\}$ there is $j(i) \in \{1, \ldots, n_i\}$ with $l_{ij(i)} = 1$. 

\item There are exactly three numbers $a,b,c \in \{0, \ldots, r\}$ such that for each $i \in \{a, b\}$ there is $j(i) \in \{1, \ldots, n_i\}$ with $l_{ij(i)} = 2$ and the numbers $l_{ik}$ are even for all $k \in \{1,\ldots, n_i\}$. Moreover, for each $i \in \{0, \ldots, r\}\setminus\{a,b,c\}$ there is $j(i) \in \{1, \ldots, n_i\}$ with $l_{ij(i)} = 1$.
\end{enumerate}

\end{theorem}

\begin{proof}

As we mentioned before, if $m>0$ then $X$ is not rigid. So further we assume that $m = 0$. 

Firstly we will show that if the condition 2) or 3) holds then $X$ is not rigid. We will use induction by~$r$. When $r = 2$ then $X$ is a trinomial hypersurface and $X$ satisfies condition 1) or 2) respectively from \cite[Theorem 2]{SG}.

Now we will prove the inductive step. Suppose that $r > 2$. For condition 1), we assume that $a = 0, b = 1$, and for condition 2), we assume that $a =0, b = 1, c =2$.

Let $\mathcal{A} = \mathbb{K}[T_{ij}|  0\leq i \leq r-1, 1 \leq j \leq n_i]$ and $Z = \mathrm{Spec}(A/(g_i | i \in I\setminus \{r-2\}) )$.

Then $X = \mathrm{Susp}(Z, f, l_{r1}, \ldots, l_{rn_r})$ for $f = p_1T_{r-2}^{l_{r-2}} + p_2T_{r-1}^{l_{r-1}}$ where
$$p_1 = -\frac{a_{0r-1}a_{1r} - a_{0r}a_{1r-1}}{a_{0r-2}a_{1r-1} - a_{0r-1}a_{1r-2}},\  p_2 = \frac{a_{0r-2}a_{1r} - a_{0r}a_{1r-2}}{a_{0r-2}a_{1r-1} - a_{0r-1}a_{1r-2}}.$$

Then $Z$ is a trinomial variety and $Z$ satisfies condition 2) or 3). So $Z$ is not rigid and by \cite[Lemma 10]{SG} $X$ is also not rigid.

Now we will show that if both conditions 2) and 3) do not hold then $X$ is rigid. If $r = 2$ the variety $X$ is a trinomial hypersurface and by \cite[Theorem 2]{SG} $X$ is rigid. So further we assume that $r > 2$. 

As per the above, $X$ is an $n_r$-suspension over $Z$. If conditions 2) and 3) do not hold for $X$ then they do not hold for trinomial varieties $X_i$ (see Section \ref{Susp}) for all $ 1 \leq i \leq n_r$. By \cite[Lemma 7]{SG}, if all $X_i$ are rigid then $X$ is rigid. So we can assume that $n_r = 1$. Similarly, we can assume that $n_0 = \ldots = n_r = 1$. Therefore, all monomials $T_i$ contain only one variable and $X$ is a surface.

In this case we will write $T_i$ instead of $T_{i1}$ and  $l_i$ instead of $l_{i1}$. If some $l_i$ equals $1$ then $X$ is isomorphic to a trinomial variety with smaller $r$. So we can assume that all $l_i > 1$.

Suppose that there are $i_1 > i_2 > i_3$ such that $l_{i_1}^{-1} + l_{i_2}^{-1} + l_{i_3}^{-1} \leq 1$. We can assume that $i_1 = 0, i_2 = 1, i_3 = 2$. Then we can apply the $ABC$-theorem for locally nilpotent derivations; see \cite[Theorem 2.48]{GF}. Then $T_0, T_1, T_2$ belong to the kernel of any locally nilpotent derivation on $X$. But this implies that all $T_i$ belong to the kernel of any locally nilpotent derivation. 

Therefore, we can assume that conditions of the $ABC$-theorem are not satisfied as well as conditions 2) and 3). Then we have five cases.

\begin{itemize}

\item[Case 1.] $l_i = 2$ for all $i$.

\item[Case 2.] There is  $i$ with $l_i > 2$ and for all $ j \in \{0, \ldots, r\}\setminus \{j\}$  we have $l_j = 2$.

\item[Case 3.] There are $i$ and $j$ with $l_i = l_j = 3$ and for all $k\in \{0, \ldots, r\}\setminus\{i, j\}$ we have $l_k = 2$.

\item[Case 4.] There are $i, j$ with $l_i =3, l_j = 4$ and for all $k \in \{0, \ldots, r\}\setminus \{i, j\}$ we have $l_k = 2$.

\item[Case 5.] There are $i, j$ with $l_i = 3, l_j = 5$ and for all $k \in \{0, \ldots, r\}\setminus \{i, j\}$ we have $l_k = 2.$  

\end{itemize}

\begin{lemma}
If a trinomial surface of Type 2 is not rigid then $X$ is rational.
\end{lemma}

\begin{proof}

Let $X$ be a trinomial surface of Type 2. Then $X$ is given by a system of equations of the following form:

$$\begin{cases}
c_1T_0^{l_0} + c_2T_1^{l_1} + c_3T_2^{l_2} = 0\\
c_4T_1^{l_1} + c_5T_2^{l_2} + c_6T_3^{l_3} = 0\\
\ldots
\end{cases},$$

A torus $T$ of dimension one acts on $X$ by the following rule:
$$t \circ T_i = t^{\frac{\mathrm{gcd}(l_0,\ldots, l_r)}{l_i}}T_i, \ t \in \mathbb{K}^*.$$

Suppose that $X$ is not rigid. Then there is a non-trivial $\mathbb{G}_a$-action on $X$. But then there is a non-trivial $\mathbb{G}_a$-action on $X$ which is normalized by $T$; see \cite[Lemma 1.10]{AL}. So the group $\mathbb{G}_a \rtimes T$ acts on $X$.

The point $O = (0, \ldots, 0) \in X$ is a unique point which is fixed with respect to the action of $T$ on $X$. Moreover, it is a singular point of $X$. Every $T$-orbit on $X$ contains $O$ in its closure. But all $\mathbb{G}_a$-orbits are closed. Let $x$ be a smooth point on $X$ that is not $\mathbb{G}_a$-stable. Then the closures of the $T$-orbit and $\mathbb{G}_a$-orbit of $x$ are distinct irreducible varieties. Therefore, the $\mathbb{G}_a\rtimes T$-orbit of $x$ is open and of dimension two.  We will denote this orbit by $U$.

Suppose that $\mathrm{St}(x) \neq \{e\}$ and consider $g\in \mathrm{St}(x)\!\setminus\!\{e\}$. Since the dimension of $U$ is equal to two, $\mathrm{St}(x)$ is a finite group and $g$ is a semisimple element of finite order. The group $\mathbb{G}_a \rtimes T$ is solvable. Hence, there is a maximal torus $\widetilde{T} \subseteq \mathbb{G}_a \rtimes T$ such that $g\in \widetilde{T}$; see \cite[Theorem 3.2.11]{VO}. But $\widetilde{T}$ is conjugated to $T$ and $T$ acts freely on the smooth points of $X$. Therefore, $\mathrm{St}(x) = \{e\}$.

It follows that $U$ is an open orbit in $X$ that is isomorphic to $\mathbb{G}_a\times T$ as a variety. So $X$ is rational.

\end{proof}

Therefore, if $X$ is not rational then $X$ is rigid. Using Theorem \ref{rational} we come to a conclusion that the following cases are only left to be checked.

\begin{itemize}

\item[a)] $r = 3, l_0 = l_1 = l_2 = 2$ and $l_3$ is an arbitrary odd number.

\item[b)] $r = 3, (l_0, l_1, l_2, l_3) = (2, 2, 3, 4).$  

\item[c)] $r = 3, (l_0, l_1, l_2, l_3) = (2, 2, 3, 5).$

\item[d)] $r = 4, (l_0, l_1, l_2, l_3, l_4) = (2, 2, 2, 3, 5).$

\end{itemize}

Let $B$ be a commutative $\mathbb{K}$-domain and $f\in B$. If $B$ is not rigid we define the \emph{absolute degree} of $f$ as
$$|f|_B = \mathrm{min}\{\mathrm{deg}_\partial(f) \mid \partial\in  \mathrm{LND}(B), \partial \neq 0\},$$
where
$$\mathrm{deg}_{\partial}(f) = \mathrm{min}\{n \in \mathbb{Z}_{\geq 0} \mid \partial^{n+1}(f) = 0\}.$$

If $B$ is rigid, we define $|f|_B = -\infty$ if $f=0$ and $|f|_B = \infty$ otherwise. We will use the following proposition.

\begin{proposition}\cite[Corollary 3.2]{FM}
Suppose $B$ is a $\mathbb{Z}$-graded affine $\mathbb{K}$-domain, $f\in B$ is homogeneous ($f\neq 0$), $\mathrm{deg}\ f \neq 0,$ and $n\geq 	2$ is an integer coprime to $\mathrm{deg}\ f$. Assume that $B' = B[z]/(f+z^n)$ is a domain. The following are equivalent. 
\begin{enumerate}
\item $|f|_{B} \geq 2$.

\item $B'$ is rigid.
\end{enumerate}
\end{proposition}

\textbf{Case a).} In this case, the variety $X$ is given by the following system of equations:

$$\begin{cases}
c_1T_0^2 + c_2T_1^2 + c_3T_2^2 = 0\\
c_4T_1^2 + c_5T_2^2 + c_6T_3^n = 0
\end{cases},$$
where $n$ is an odd number and $c_i \neq 0$ for all $i$.

We denote by $\widetilde{R}$ the algebra $\mathbb{K}[T_0, T_1, T_2]/(c_1T_0^2+ c_2T_1^2 + c_3T_2^2)$. We will apply \cite[Corollary 3.2]{FM} to the algebra $\widetilde{R}$ and $f = \frac{c_4}{c_6}T_1^2 + \frac{c_5}{c_6}T_2^2$. There is a grading on $\widetilde{R}$ such that $\mathrm{deg}\ T_0 = \mathrm{deg}\ T_1 = \mathrm{deg}\ T_2 = 1$. Then $f$ is homogenous and $\mathrm{deg}\ f = 2$. 
Therefore, if we show that $|f|_{\widetilde{R}} \geq 2$ then $X$ will be rigid.

Suppose that there is a non-zero $\partial \in \mathrm{LND}(\widetilde{R})$ such that $\partial(f) = 0$. Then by \cite[Theorem 2.50]{GF} we have $\partial(T_1) = \partial(T_2) = 0$. But it implies that $\partial(T_0) = 0$ and $\partial = 0$. We obtain a contradiction. So $|f|_{\widetilde{R}} \geq 1.$

Now suppose that there is a non-zero $\partial \in \mathrm{LND}(\widetilde{R})$ with $\partial^2(f) = 0$. We have
$$f = \frac{c_4}{c_6}T_1^2 + \frac{c_5}{c_6}T_2^2 = \left( \beta_1 T_1 + \beta_2T_2\right)\left(\beta_1 T_1 - \beta_2 T_2\right)$$
where $\beta_1^2 = \frac{c_4}{c_6}$ and $\beta_2^2 = -\frac{c_5}{c_6}$. Since $\mathrm{deg}_{\partial}\ f = 1$ then $\mathrm{deg}_\partial(\beta_1 T_1 + \beta_2T_2) = 0$ or $\mathrm{deg}_\partial(\beta_1 T_1 - \beta_2T_2) = 0.$ We can assume that $\partial(\beta_1 T_1 + \beta_2T_2) = 0$. This implies that $\partial(T_2) = -\frac{\beta_1}{\beta_2}\partial(T_1)$.

On the other hand
$$f = \frac{c_4}{c_6}T_1^2 + \frac{c_5}{c_6}T_2^2 =  \frac{c_4}{c_6}T_1^2 +  \frac{c_5}{c_6}\left(-\frac{c_1}{c_3}T_0^2  - \frac{c_2}{c_3}T_1^2\right) = \left(\gamma_0T_0 + \gamma_1 T_1\right)(\gamma_0 T_0 - \gamma_1T_1) $$ 
where $\gamma_0^2 = -\frac{c_5}{c_6}\frac{c_1}{c_3}, \gamma_1^2 = -\frac{c_4}{c_6} + \frac{c_5}{c_6}\frac{c_2}{c_3}.$ Again we can assume that $\partial\left(\gamma_0T_0 + \gamma_1 T_1\right) = 0$. So $\partial(T_0) = -\frac{\gamma_1}{\gamma_0}\partial(T_1).$

Since $c_1T_0^2 + c_2T_1^2 + c_3T_2^2 = 0$ in $\widetilde{R}$ we have 
$$ 0 = \partial(c_1T_0^2 + c_2T_1^2 + c_3T_2^2) = 2c_1T_0\partial(T_0) + 2c_2T_1\partial(T_1) + 2c_3T_2\partial(T_2) = $$
\begin{equation}\label{eq1}
2(-c_1\frac{\gamma_1}{\gamma_0}T_0 + c_2T_1 -c_3\frac{\beta_1}{\beta_0}T_2)\partial(T_1).
\end{equation}
The elements $T_0, T_1, T_2$ are linearly independent in $\widetilde{R}$ and since $c_2 \neq 0 $ we have 
$$2(-c_1\frac{\gamma_1}{\gamma_0}T_0 + c_2T_1 -c_3\frac{\beta_1}{\beta_0}T_2) \neq 0.$$
Then equation \ref{eq1} implies $\partial(T_1) = 0$. Since $\beta_1 \neq 0$ we have $\partial(T_2) = 0$. Finally, 
$$\gamma_1^2 = -\frac{c_4}{c_6} + \frac{c_5}{c_6}\frac{c_2}{c_3} = \frac{1}{c_6c_3}\left( -c_4c_3 + c_5c_2\right) =$$

$$ \frac{1}{c_6c_3}\left( -(a_{02}a_{13} - a_{03}a_{12})(a_{00}a_{11} - a_{01}a_{10}) - (a_{01}a_{13} - a_{03}a_{11})(-a_{00}a_{12} + a_{02}a_{10}\right) = $$
$$\frac{1}{c_6c_3}\left(-a_{00}a_{02}a_{11}a_{13} - a_{01}a_{03}a_{10}a_{12}  + a_{00}a_{01}a_{12}a_{13} + a_{02}a_{03}a_{10}a_{11}\right) =$$
$$\frac{1}{c_6c_3}\left(a_{00}a_{13}(a_{01}a_{12} - a_{02}a_{11}) - a_{03}a_{10}(a_{01}a_{12} - a_{02}a_{11})\right) = $$
$$\frac{1}{c_6c_3}\left(a_{00}a_{13} - a_{03}a_{10}\right)\left(a_{01}a_{12} - a_{02}a_{11}\right).$$
Since the matrix $A$ from Definition \ref{defTrin} has pairwise linearly independent columns we obtain $\gamma_1 \neq 0 $. Since $\partial(T_0) = -\frac{\gamma_1}{\gamma_0}\partial(T_1) = 0$ we obtain that $\partial = 0$. So $|f|_{\widetilde{R}} \geq 2$. Then by \cite[Corollary 3.2]{FM} $X$ is rigid.

\textbf{Case b).}

In this case $X$ is given by the following system of equations:

$$\begin{cases}
c_1T_0^2 + c_2T_1^2 + c_3T_2^4 = 0\\
c_4T_1^2 + c_5T_2^4 + c_6T_3^3 = 0
\end{cases},$$
where $c_i \neq 0$ for all $i$. 

Again, we will use \cite[Corollary 3.2]{FM}. We take $\widetilde{R} = \mathbb{K}[T_0, T_1, T_2]/(c_1T_0^2 + c_2T_1^2 + c_3T_2^4)$ and $f = \frac{c_4}{c_6}T_1^2 + \frac{c_5}{c_6}T_2^4.$ We consider the grading on $\widetilde{R}$ with $\mathrm{deg}\ T_0 = \mathrm{deg}\ T_1 = 2$ and $\mathrm{deg}\ T_2 = 1$.  Therefore, $\mathrm{deg}\ f = 4$, which is coprime to $3$. 

Suppose that there is a locally nilpotent derivation $\partial$ on $\widetilde{R}$ such that $\partial(f) = 0$ or $\partial^2(f) = 0$. Then, by \cite[Theorem 2.50]{GF}, $\partial(T_1) = \partial(T_2) = 0$, which implies $\partial(T_0) = 0. $ So $\partial = 0$. Hence, $|f|_{\widetilde{R}} \geq 2$ and by \cite[Corollary 3.2]{FM} $X$ is rigid.

\textbf{Case c).} 

In this case we can assume that $X$ is given by the following system of equations:

$$\begin{cases}
c_1T_0^2 + c_2T_1^2 + c_3T_2^3 = 0\\
c_4T_1^2 + c_5T_2^3 + c_6T_3^5 = 0
\end{cases},$$
where $c_i \neq 0$ for all $i$. 

We consider $\widetilde{R} = \mathbb{K}[T_0, T_1, T_2]/(c_1T_0^2 + c_3T_1^2 + c_3T_2^3)$,\ $f = \frac{c_4}{c_6}T_1^2 + \frac{c_5}{c_6}T_2$, and the grading on $\widetilde{R}$ with $\mathrm{deg}\ T_0 = \mathrm{deg}\ T_1 = 3,\ \mathrm{deg}\ T_2 = 2.$ Then $\mathrm{deg}\ f = 6$, which is coprime to $5$. 

If there is a locally nilpotent derivation on $\widetilde{R}$ with $\partial(f) = 0$ or $\partial^2(f) = 0$ then by \cite[Theorem 2.50]{GF} we have $\partial(T_1) = \partial(T_2) = 0$, which implies that $\partial(T_0) = 0$ and $\partial = 0$. So $X$ is rigid by \cite[Corollary 3.2]{FM}.

\textbf{Case d).}

In this case we can assume that $X$ is given by the following system of equations: 

$$\begin{cases}
c_1T_0^2 + c_2T_1^2 + c_3T_2^2 = 0\\
c_4T_1^2 + c_5T_2^2 + c_6T_3^3 = 0\\
c_7T_2^2 + c_8T_3^3 + c_9T_4^5 = 0.
\end{cases},$$
where $c_i \neq 0$ for all $i$. 

We consider the ring $\widetilde{R} = \mathbb{K}[T_0, T_1, T_2, T_3]/(c_1T_0^2 + c_2T_1^2 + c_3T_2^2, c_4T_1^2 + c_5T_2^2 + c_6T_3^3 )$, the element $f = \frac{c_7}{c_9}T_2^2 + \frac{c_8}{c_9}T_3^3$, and the grading with $\mathrm{deg}\ T_0 = \mathrm{deg}\ T_1 = \mathrm{deg}\ T_2 = 3,\ \mathrm{deg}\ T_3 = 2$. Then $f$ is homogeneous and $\mathrm{deg}\ f= 6$, which is coprime to $5$. 

By case a) the ring $\widetilde{R}$ is rigid. So $|f|_{\widetilde{R}} = \infty > 2$ and $X$ is rigid by   \cite[Corollary 3.2]{FM}.

 \end{proof}

Combining Theorems \ref{Type1} and \ref{Type2} we obtain the main theorem.

\begin{theorem}\label{MainTheorem}

Let $X$ be a trinomial variety of Type 1. Then $X$ is not rigid if and only if one of the following holds:
\begin{enumerate}
\item $m > 0$;
\item There is $b \in \{1,\ldots, r\}$ such that for each $ i\in \{1, \ldots, r\} \setminus \{b\}$ there is $j(i) \in \{1, \ldots, n_i\}$ with $l_{ij(i)} = 1$.

\end{enumerate}

Let $X$ be a trinomial variety of Type 2. Then $X$ is not rigid if and only if one of the following holds:
\begin{enumerate}

\item $m > 0$;

\item There are at most two numbers $a, b \in \{0,\ldots, r\}$ such that for each $i \in \{0,\ldots, r\} \setminus \{a, b\}$ there is $j(i) \in \{1, \ldots, n_i\}$ with $l_{ij(i)} = 1$. 

\item There are exactly three numbers $a,b,c \in \{0, \ldots, r\}$ such that for each $i \in \{a, b\}$ there is $j(i) \in \{1, \ldots, n_i\}$ with $l_{ij(i)} = 2$ and the numbers $l_{ik}$ are even for all $k \in \{1,\ldots, n_i\}$. Moreover, for each $i \in \{0, \ldots, r\}\setminus\{a,b,c\}$ there is $j(i) \in \{1, \ldots, n_i\}$ with $l_{ij(i)} = 1$.
\end{enumerate}

\end{theorem}

\end{document}